\documentclass[12 pt]{amsart}

\usepackage{amssymb,amscd, hyperref }
\usepackage[left=3.3cm,right=3.3cm]{geometry}
\usepackage[mathcal]{eucal}
\usepackage{array,float}

\usepackage{xy}
\input xy
\xyoption{all}
\usepackage{pdflscape}
\usepackage[draft]{graphicx}

\numberwithin{equation}{section}

\newtheorem{thm}{Theorem}[section]
\newtheorem{proposition}[thm]{Proposition}

\newtheorem{lemma}[thm]{Lemma}

\theoremstyle{definition}
\newtheorem{definition}[thm]{Definition}
\newtheorem*{remark*}{Remark}
\newtheorem{remark}[thm]{Remark}

\newtheorem{example}[thm]{Example}

\newcommand{\hf}{(H \cap F)}

\newcommand{\h}{\mathbf{h}}

\newcommand{\End}{\text{End}}

\newcommand{\ind}{\mathrm{Ind}}

\newcommand{\hirschhigh}[2]{\langle #1, #2 \rangle}

\title[On representations of supersolvable groups]{On characterization of monomial representations of discrete supersolvable groups}
\date{\today }
\author{E. K. Narayanan}
\address{ Department of Mathematics,
    Indian Institute of Science,
    Bangalore 560012, India }

\email{naru@math.iisc.ernet.in}
\author{Pooja Singla}

\email{pooja@math.iisc.ernet.in}

\subjclass[2010] { Primary: 20C15,
Secondary: 20F18, 22D30}

\keywords{ Monomial representations, Induced
representations, supersolvable groups, finitely generated nilpotent groups, Infinite dihedral groups, Hirsch length}
\begin{document}
\maketitle
\begin{abstract}
 We prove that an abstract (possibly infinite dimensional) complex irreducible representation of a discrete supersolvable group is monomial if and only if it has finite weight. We also prove a general result that implies converse of Schur's lemma holds true for certain induced representations of finitely generated discrete groups. At last, we work out example of infinite dihedral group and prove that it is a monomial group. 

\end{abstract}

\maketitle

\section{Introduction}
A representation of a group is called monomial if it is induced from a one dimensional representation and a group is called monomial if all of its irreducible representations are monomial.
It is a classical result that a finite nilpotent group is monomial (see Serre~\cite[Ch.8]{MR0450380}). The proof of this result boils down to the existence of a non-central normal abelian subgroup. The latter property holds for the more general class of groups called supersolvable groups. Recall a group is called supersolvable if there exists a finite invariant normal series such that each quotient is a cyclic group. Therefore it can be deduced that every irreducible representation of a finite supersolvable group is monomial (see Serre~\cite[Ch.8]{MR0450380}).

For the case of simply connected nilpotent Lie groups and complex unitary irreducible representations the analogous result appeared in 1962 in the classical work of Kirillov~\cite{Kirillov1962} (see also Dixmier~\cite{Dixmier3, Dixmier4}), where he developed the famous orbit method to construct irreducible representations of simply connected nilpotent Lie groups. It turns out that every unitary irreducible representation of a simply connected
connected nilpotent Lie group is induced from a unitary character of a subgroup. In the same article, Kirillov mentioned that many of his proofs work for a more general class of simply connected completely solvable Lie groups, that is, the groups that have an invariant normal series such that each quotient is a one dimensional Lie group. In particular every simply connected completely solvable Lie groups is monomial. For results regarding extension of orbit method to completely solvable Lie groups, we refer the reader to Lipsman~\cite{Lipsman(1)90}. The orbit method of Kirillov can be extended to more general setting of exponentially solvable Lie groups (see Fujiwara~\cite{Fujiwara90, Fujiwara2015}). As a consequence one obtains that exponential solvable Lie groups are monomial (see Bernat et al.~\cite{Bernat}, Takenouchi~\cite{Takenouchi}).

On the other hand, the complex irreducible unitary representations of finitely generated discrete nilpotent groups are not necessarily monomial. For example, I.D.Brown~\cite{MR0352324} constructed an irreducible unitary representation of discrete Heisenberg group  that is not monomial. Brown~\cite{MR0352324} also proved that a unitary irreducible representations of a discrete nilpotent group is monomial if and only if it has finite weight. Recall that a representation $\rho$ of $G$ is said to have finite weight if there exists a subgroup $H$ of $G$ and a character $\chi$ of $H$ such that space of $H$-linear maps, also called the space of intertwining operators, $\mathrm{Hom}_H(\rho|_{H}, \chi)$ is non-zero and finite dimensional.

Brown's result~\cite{MR0352324} motivates the question whether a similar characterization holds for irreducible representations (unitary
or otherwise) of finitely generated discrete nilpotent groups. That is, if $G$ is such a group and $\pi$ is an irreducible representation (not necessarily
unitary) of $G,$ then is it true that $\pi$ is equivalent to $\ind_{H}^G(\chi)$ for some character $\chi$ of a subgroup $H$ if and only if $\pi$ has
finite weight (see the conjecture by Parshin in \cite{ParshinICM}, see also~\cite{Arnal-Parshin}). We remark that, by $\ind_H^G(\chi)$ we mean the finite induction (see Definition~\ref{def:finite-induction}). Very recently Beloshapka-Gorchinskiy~\cite{Beloshapka-Gorchinskiy} proved this conjecture for finitely generated discrete nilpotent groups. See also \cite{Nilpotent-1} by authors for a slightly different approach to this problem.


In this paper our aim is to extend the characterization of monomial representations given by Brown~\cite{MR0352324} to the class of discrete supersolvable groups.
It is not very difficult to prove that every finite dimensional representation of a discrete supersolvable group is monomial (see Proposition~\ref{prop: finite-dimensional}). So our main attention will be on infinite dimensional representations. Note that a nilpotent group is supersolvable if and only if it is finitely generated. Hence our results extend the analogous results for finitely generated nilpotent groups. Below we describe the main results of this article.

Let $(\pi, V)$ be an irreducible representation of discrete supersolvable group $G$. Let $H$ be a subgroup of $G$ and $\chi$ be a character of $H$ such that the following subspace of $V$ is non-zero and finite dimensional.
\[
V_H(\chi) = \{ v \in V \mid \pi(h) v = \chi(h)v \,\, \forall \,\, h \in H \}.
\]
In this case, we say $\pi$ has {\it finite weight} with respect to $(H, \chi)$. In case $(\pi, V)$ is an irreducible representation of $G$ with finite weight, then our aim is to find
a subgroup $H'$ of $G$ and a character $\chi'$ of $H'$ such that $\pi \cong \ind_{H'}^G(\chi')$. The plan of the paper is as follows: In section~\ref{sec: preliminaries}, we collect few required known facts  about supersolvable groups and basic representation theory for the reader's convenience.

Recall that a representation $V$ of $G$ is said to be Schur irreducible if $\mathrm{End}_G(V)$ is one dimensional. It is clear that if $V$ is irreducible then $V$ is Schur irreducible. The converse clearly holds for
unitary representations. However, this does not hold for non-unitary representations (see~\cite{Beloshapka-Gorchinskiy}). Also recall that for a subgroup $H$ of a group $G$, the radical of $H$ in $G$, denoted by $\sqrt[G]{H}$, is the set $\{ g \in G \mid g^i \in H\,\,\mathrm{for} \,\, \mathrm{some}\,\, i \in \mathbb Z \}$ (see Definition~\ref{def:radical} ). In case $G$ is finitely generated nilpotent group then $\sqrt[G]{H}$ for any subgroup $H$ of $G$ is a group. However, this need not be true for general $G$. In Section~\ref{sec: kutzko}, we prove the following important sufficient condition for $\ind_H^G(\chi)$ to be Schur irreducible for a subgroup $H$ of a supersolvable group $G$.


\begin{thm}\label{thm: reduction-to-normalizer} Let $G$ be a discrete supersolvable group and let $F$ denote its Fitting subgroup. Let $H$ be a subgroup of $G$ and let $\chi$ be a character of $H$ such that $\chi^g \neq \chi$ on ${H^g \cap H}$ for all $g \in N_G(\sqrt[F]{H\cap F}) \setminus H$. Then $\mathrm{Ind}_H^G(\chi)$ is Schur irreducible.
\end{thm}

In Section~\ref{sec: existence}, we use the above result for the case when $(\pi, V)$ is an irreducible representation of supersolvable $G$ having finite weight with respect to $(H, \chi)$. We show that there exists a subgroup $H'$ and its character $\chi'$ such that the induced representation $\ind_{H'}^G(\chi')$ is Schur irreducible and $V_{H'}(\chi')$ is non-trivial. In the proof we see that finite weight condition is used crucially. The other fact that turns out to be very important is the existence of a normal subgroup $N$ of an infinite supersolvable groups $G$ such that $G/N$ is an infinite dihedral group (see Lemma~\ref{lem:index dihedral}). We remark that in case all infinite subgroups of $G$ contain normal subgroups such that quotient is infinite cyclic then proof turns out to be easier and this is the case of finitely generated nilpotent groups.

Section~\ref{sec: irreducible} is devoted to proving that the obtained induced representation $\ind_{H'}^{G}(\chi')$ is irreducible because $G$ is a supersolvable group. Recall that Schur irreducibility need not
imply irreducibility and so Theorem~\ref{thm: reduction-to-normalizer} does not immediately imply irreducibility of $\ind_{H'}^G(\chi').$ However, we show that this is indeed the case by establishing a general result
which may be of independent interest. Before we state this result we define {\it proximate normal} subgroups.

A subgroup $H$ of $G$ is called {\it proximate normal} if for all $g \in G$, the groups $H \cap H^g$ are finite index normal subgroups of $H$ and the quotient $H/(H \cap H^g)$ is abelian. This is a generalization of the notion of a normal subgroup.

\begin{thm}
\label{thm: general-irreducibility-result}
	Let $H$ be a proximate normal subgroup of a discrete group $G$ and let $\rho$ be an irreducible representation of  $H$ such that $\ind_H^G(\rho)$ is Schur irreducible. Then $\ind_H^G(\rho)$ is an irreducible $G$-representation.
\end{thm}
The above result for the case when $H$ is normal was obtained in \cite{Beloshapka-Gorchinskiy}. For the case of supersolvable groups, we define a notion of subproximate normal subgroup (see Definition~\ref{defn:subproximate-normal}) and show that every subgroup of a supersolvable group is subproximate. We use this with the above general result to prove the following for supersolvable groups.
\begin{thm}
\label{thm: supersolvable-irreducible}
Let $G$ be a supersolvable group. Let $H$ be a subgroup of $G$ and let $\chi$ be a character of $H$ such that $\ind_H^G(\chi)$ is Schur irreducible. Then the induced representation $\ind_H^G(\chi)$ is an irreducible $G$-representation.
\end{thm}

In section~\ref{sec: main-theorem}, we establish the main result of the article.

\begin{thm}\label{main-theorem}
Let $G$ be a discrete supersolvable group. An irreducible representation of $G$ is monomial if and only if it has
finite weight.
\end{thm}

This is achieved by establishing a non-trivial $G-$linear map between $\pi$ (given finite weight representation) and $\ind_{H'}^{G}(\chi')$ which is irreducible by Theorem~\ref{thm: supersolvable-irreducible}. In section~\ref{sec: infinite-dihedral}, we discuss the example of infinite dihedral group $(D_\infty)$ (see Definition~\ref{eg: infinite-dihedral}) in detail and prove the following.
\begin{thm}
	The infinite dihedral group is monomial.
\end{thm}

\section{Preliminaries}
\label{sec: preliminaries}
In this section we fix notation and recollect few facts that we use later in the results.

\subsection{Results from Group Theory:}
We will use the following result very frequently.
 \begin{lemma}(\cite[Section~2]{Hall1950})
 	\label{finite-finite.index}
 	For a finitely generated group $G$, there exists only finitely many subgroups of a given index.
 	
 \end{lemma}

Recall a group $G$ is called {\it supersolvable}, if there exist $k \in \mathbb N$ and normal subgroups $N_i$ of $G$ for $1 \leq i \leq k$ such that,
\begin{equation}
\label{eq:invarint-series}
1=N_{k+1} \subseteq N_{k} \subseteq N_{k-1} \cdots \subseteq N_0 = G,
\end{equation}
and for all $0 \leq i \leq k$ groups $N_i/N_{i+1}$ are cyclic.

The following examples of supersolvable groups may be useful for the reader to keep in mind while working out the details.
\begin{example}(Infinite dihedral Group($D_\infty$))
	\label{eg: infinite-dihedral}
	This is a generalization of finite dihedral groups. As a set $D_\infty$ consists of tuples $(s, z)$, where $s = \pm 1$ and $z \in \mathbb Z$ with multiplication as given below,
	\[
	(s, z)(s', z') = (ss', s'z + z').
	\]
It is easy to see that $D_\infty$ is a supersolvable group that is not nilpotent.
\end{example}
\begin{example}(Generalized discrete Heisenberg group) We shall use $GH_3(\mathbb Z)$ to denote the following set.
\[
GH_3(\mathbb Z) = \{ \left[ \begin{matrix} a & x & z \\ 0 & b & y \\ 0 & 0 & c \end{matrix} \right] \mid a,b,c \in \{ \pm1 \} \,\, \mathrm{and}\,\, x,y,z \in \mathbb Z   \}
\]
Then $GH_3(\mathbb Z)$ is a group under matrix multiplication and we call this generalized discrete Heisenberg group. This is also an example of a supersolvable group that is not nilpotent.
\end{example}

Note that a supersolvable group is finitely generated by definition. As subgroups of supersolvable group are supersolvable, it follows that every subgroup of a supersolvable group is finitely generated. As mentioned earlier, and is not very difficult to show, that a nilpotent group is supersolvable if and only if it is finitely generated. The following result says the relation between supersolvable and nilpotent is deeper.

\begin{definition}(Fitting subgroup) The Fitting subgroup of $G$ is the subgroup of $G$ generated by all normal nilpotent subgroups of $G$.
\end{definition}
\begin{lemma}(\cite[Theorem~1.10]{Pinnock98})
\label{fitting-subgroup}
Let $G$ be a supersolvable group and $F$ be its Fitting subgroup. Then the following are true.
\begin{enumerate}
\item The group $F$ is a nilpotent normal subgroup of $G$.
\item The group $G/F$ is a finite abelian subgroup.
\end{enumerate}
\end{lemma}

%
%
%

It is well known that for any given invariant series of a supersolvble group $G$ of the form (\ref{eq:invarint-series}), the number of $i$ such that $N_i/N_{i+1}$ is an infinite cyclic group is an invariant of the group $G$ (see \cite[5.4.13]{Robinson96}). This invariant is called {\it Hirsch length} of $G$ and is denoted by $\h(G)$. If $N$ is a normal subgroup of $G$ then $N$ has finite index in $G$ if and only if $\h(N) = \h(G).$
In our proofs we will use induction on $\h(G).$  An example of such a result is the following which is crucial for us.

\begin{lemma}(\cite[Lemma 3.8]{MR470211})
\label{lem:index dihedral} Let $G$ be an infinite supersolvable group. Then $G$ has a normal subgroup $N$ such that $G/N$ is either an infinite cyclic group or an infinite dihedral group.
\end{lemma}

Other properties of subgroups of supersolvable group, that we will use, are given by the following:

\begin{lemma} (see \cite[5.4.11, 5.4.15]{Robinson96})
\label{lem: infinite-cyclic}
 \begin{enumerate}
 \item A torsion subgroup of a supersolvable group is finite.
 \item An infinite supersolvable group contains an infinite cyclic normal subgroup.
\end{enumerate}
\end{lemma}

\subsection{Results from Representation Theory:}
Let $G$ be a topological group. By a  representation of $G$ we mean a group
homomorphism from $G$
to the set of automorphisms of a complex vector space $V$.
We denote this by $(\pi, V)$. Whenever meaning is clear from the context, we also denote representation by $\pi$ or $V$ only. By a character of a group we always mean its one dimensional representation.

For our case we will only deal with
cases where $G$ is a discrete group. For this case, we use finite induction (also called compact induction) as defined below.

\begin{definition}(Induced representation)
\label{def:finite-induction}
	Let $H$ be a subgroup of a discrete group $G$ and $(\rho, W)$ be a
	representation of $H$. The induced representation
	$(\widetilde{\rho}, \widetilde{W})$ of $\rho$ from $H$ to $G$ has
	representation space $\widetilde{W}$ consisting of functions $f: G
	\rightarrow \End(W)$ satisfying the following:
	\begin{enumerate}
		\item $f(hg) = \rho(h^{-1}) f(g)$ for all $g \in G$, $h \in H$;
		\item Support of $f$ is contained in a set of finite number of
		right cosets of $H$ in $G$,
	\end{enumerate}
	and the homomorphism $\widetilde{\rho}: G \rightarrow
	\mathrm{Aut}(\widetilde{W})$ given by $\widetilde{\rho}(g)f(x)
	= f(xg)$ for all $x, g \in G$. We denote this induced
	representation by $\ind_H^G(\rho)$.
\end{definition}

Further recall, a representation $(\pi, V)$ of a group $G$ is said to have {\bf
        finite weight} if there is a subgroup $H$ and a character $\chi: H
    \rightarrow \mathbb C^{\times}$ such that the space $V_H(\chi)$
    defined by,
    \[
    V_H(\chi) = \{v \in V \mid \pi(h) v = \chi(h) v\,\, \mathrm{ for \,\, all}\,\,  h \in H \},
    \]
    is finite dimensional. In this case we say that $\pi$ has {\it finite
    weight} with respect to $(H, \chi)$. We also need the following results.
 \begin{proposition}\label{prop: reduction}
 	
 	Let $(\pi, V)$ be an irreducible representation of $G$ such that
 	$V_H(\chi) \neq 0$. Then there exists a non-trivial $G$-linear map, also called intertwining operator, between $G$-spaces $\mathrm{Ind}_H^G(\chi)$ and $\pi$.
 \end{proposition}
 \begin{proof}
Let $\rho$ denote the representation $\ind_H^G(\chi).$
 	Let $X = \{g_i \mid i \in \mathbb I \}$ be the right coset
 	representatives of $H$ in $G$. For a fixed $v \in V_H(\chi)$, let
 	$W$ be the one dimensional space generated by $v$. We define
 	functions $f_i: X \rightarrow W$ by $f_i(g_j) = \delta_{i,j}(v)$ and
 	extend these to $G$ so that $f_i \in \mathrm{Ind}_H^G(\chi)$ for
 	all $i \in \mathbb I$. Then any $f \in \mathrm{Ind}_H^G(\chi)$ can
 	be written as linear combination of $f_i$'s. For any $g \in G$, we
 	have
 	\begin{eqnarray}\label{g-action}
 	\rho(g)f_i(hg_j) = f_i(hg_jg) = \chi(h) f_i(g_jgg_{i}^{-1}g_i)
 	\end{eqnarray}
 	
 	Therefore $\rho(g)f_i$ is nonzero on $g_j$ for $g_j$ satisfying
 	$g_jg \in H{g_i}$. Thus $\rho(g)f_i = \chi(g_jgg_{i}^{-1})f_j$ for
 	$j$ such that $g_jg \in H{g_i}$. Now define $F:
 	\mathrm{Ind}_H^G(\chi) \rightarrow \pi$ by $F(f_i) =
 	\pi(g_i^{-1})v$ on $f_i$'s and extended linearly thereof. Then,
 	\[
 	F(\rho(g)f_i) = F(\chi(g_jgg_{i}^{-1}) f_j) = \chi(g_jgg_{i}^{-1})\pi(g_j^{-1})v = \pi(g) \pi(g_i^{-1})v.
 	\]
 	This shows that $F \in \mathrm{End}_G(\mathrm{Ind}_H^G(\chi), \pi)$ is a
 	non-zero intertwiner.
 \end{proof}


\begin{lemma}
\label{diamond-lemma}

Let $G$ be a group. Let $H$ and $K$ be two subgroups of $G$
such that $K$ normalizes $H$. Let $\chi$ be a one dimensional representation of $H$ that is stable under the conjugation action of $K$. Let $\delta$ be a one dimensional representation of $K$ such that $\chi|_{H \cap K} = \delta|_{H \cap K}$.
%
%
%
%
%
Then $\tilde{\chi}: HK \rightarrow \mathbb C^{\times}$ defined by
$\tilde{\chi}(hk) = \chi(h) \delta(k)$ for all $h \in H$ and $k \in
K$ is a character of $HK$ such that $\tilde{\chi}|_{H} = \chi$.

\end{lemma}

\begin{proof}
	The one dimensional  representation $\tilde{\chi}$ is well defined because $h_1 k_1 = h_2 k_2$ implies $h_2^{-1} h_1 = k_2 k_1^{-1} \in H \cap K$. Therefore $\chi(h_2)^{-1} \chi(h_1) = \delta (k_2) \delta(k_1)^{-1}$. The rest of the proof follows easily.
\end{proof}

\section{Creterion for Schur irreducibiity of induced representation}
\label{sec: kutzko}

In this section, we deduce a sufficient condition on pair $(H, \chi)$ that implies that the
induced representation $\ind_H^G(\chi)$ is Schur irreducible.

Let $G$ be a finitely generated discrete group and $H$
be a subgroup of $G$. Let $(\chi, W)$ be a one dimensional representation of $H$. For $g \in G$, let
$\chi^{g} : g^{-1}Hg \rightarrow \mathbb C^\times$ be the one dimensional representation of $g^{-1}Hg$, defined by $\chi^g(g^{-1}hg) = \chi(h)$.
By $\mathrm{Hom}_{H \cap g^{-1}Hg }(\chi, \chi^{g})$, we mean the
space of $(H \cap (g^{-1}Hg))$-linear maps from $(\chi, W)$ to
$(\chi^{g}, W)$ when both of these are viewed as
representations of the group $H \cap g^{-1}Hg$. The following result regarding the space of intertwining operators of induced representations is due to Kutzko~\cite{MR0442145}. Its unitary analogue was proved by
Mackey~\cite{MR0042420} in early 50's.

\begin{thm}(Kutzko~\cite{MR0442145})\label{thm:kutzko} Let $G$ be a discrete group and let $H$ be a subgroup of $G$. Let $(\chi, W)$ be a one dimensional representation of $H$. Then the following are true.
	\begin{enumerate}
		
		\item  The space $\mathrm{End}_G(\ind_H^G(\chi))$
		is isomorphic to the following space of functions.
		\[
		\Delta = \left\{
		\begin{array}{ccc}
	s : G \rightarrow \mathrm{End}_{\mathbb C}(W) &  \mid &
	\begin{array}{c} s(h_1gh_2) = \chi(h_1) s(g) \chi(h_2), \\ \forall \,\, h_1, h_2 \in H, \,\forall \,\, g
	\in G \end{array}
		\end{array}
		\right\}.
		\]
		\item  $\mathrm{End}_G(\ind_H^G(\chi)) \cong \oplus_{g \in H \backslash G /H} \mathrm{Hom}_{H \cap g^{-1}Hg }(\chi, \chi^{g})$
		
	\end{enumerate}
	
\end{thm}

\begin{proof} Here we outline some ideas of the proof which we will use in the proof of Theorem~\ref{thm: reduction-to-normalizer}. For details see(\cite{MR0442145}, p.2).
	For any $w \in W$, define the set of functions $f^w:G \rightarrow
	W$ by the following.
	\[
	f^w(x) = \begin{cases} \chi(x)w & \text{for} \,\, x \in H \\ 0  & \text{otherwise}. \end{cases}
	\]
	Then
	$f^w \in \ind_H^G(\chi)$. For $\psi \in  \mathrm{End}_G(\ind_H^G(\chi))$, define $ s_\psi: G \rightarrow \mathrm{End}_{\mathbb C}(W)$ by  $s_\psi(g)(w) =
	\psi f^w(g).$
	
	Then $\Psi: \psi \mapsto s_\psi$ gives the required isomorphism
	between $\mathrm{End}_G(\ind_H^G(\chi))$ and $\Delta$. As for the second part, for each $g \in H \backslash G /H $, let $\Delta_{g}$ consists of
	functions in $\Delta$ that have their support on $HgH$. Then
	$\Delta \cong \oplus_{g \in H \backslash G /H } \Delta_{g}$.
	
	For all $g \in H \backslash G /H$, we have
	$\Delta_{g}$ is isomorphic to  $\mathrm{Hom}_{H \cap g^{-1}Hg }(\chi, \chi^{g})$ with isomorphism given by $s \mapsto
	s(g)$.
\end{proof}
\begin{remark} 
\label{rk:general-intertwiner} In general, when  $\rho$ is any representation of a subgroup $H$ of $G,$ the following result is true by {\it Mackey's formula} (see \cite[Chapter~1, Section~5.5, 5.7]{Vigneras96}, also see \cite{Beloshapka-Gorchinskiy}.
\begin{eqnarray}
\label{eq:mackey-formula}
\ind_H^G(\rho)|_{H} \cong \oplus_{g \in H \backslash G /H } \ind_{H \cap H^g}^H(\rho^g|_{H \cap H^g}).
\end{eqnarray}
This gives the following description of $G$-linear maps of $\ind_H^G(\rho)$. 
\begin{eqnarray}
\mathrm{End}_G(\ind_H^G(\rho)) \cong \oplus_{g \in H \backslash G/  H} \mathrm{Hom}_H(\rho, \ind_{H^g \cap H}^H(\rho^g|_{H^g \cap H})  ).
\end{eqnarray}
Further in case the index of $H$ in $G$ is finite and $\pi$ is an arbitrary representation of $G$, then the following vector spaces are canonical isomorphic (see \cite[Chapter~1, Section~5.4]{Vigneras96}).
\begin{eqnarray}
\label{frobenius-reciprocity}
\mathrm{Hom}_G(\pi, \ind_H^G(\rho) \cong \mathrm{Hom}_H(\pi|_H, \rho).
\end{eqnarray}

\end{remark}

\begin{definition}(Radical)
\label{def:radical} The radical of a subgroup $H$ of $G$, denoted by $\sqrt[G]{H}$, is the set given by the following.
\[
\sqrt[G]{H} = \{g \in G \mid g^n \in H\,\, \mathrm{for\,\, some}\,\, n \in \mathbb N \}.
\]
\end{definition}
Note that radical need not be a group in general. However situation is nice for the case of finitely generated nilpotent groups. The following facts about radical are crucial for us.
\begin{lemma}
\label{lem: baumslag} Let $K$ be a finitely generated nilpotent group.
 Then for any subgroup $H$ of $K$, the following results are true.
	\begin{enumerate}
		
		\item The set $\sqrt[K]{H}$ is a subgroup of $K$ and index $[\sqrt[K]{H} : H] < \infty$.
		\item $\sqrt[K]{(\sqrt[K]{H})} = \sqrt[K]{H}$.
		\item $\sqrt[K]{N_K(H)} = N_K(\sqrt[K]{H})$.
\end{enumerate}
	
\end{lemma}
\begin{proof}
	Proof follows from Baumslag~\cite[Lemma~2.8]{MR0283082} and Brown~\cite[Lemma 4]{MR0352324}.
\end{proof}
We remark that above results do not hold for the class of supersolvable groups. Next we prove Theorem~\ref{thm: reduction-to-normalizer}.

For the case when $G$ is supersolvable, the following gives a sufficient condition on the pair $(H, \chi)$ such that $\ind_H^G(\chi)$ is Schur irreducible.


\begin{proof}[Proof of Theorem~\ref{thm: reduction-to-normalizer}]By Theorem~\ref{thm:kutzko}, it is enough to prove that if $g\in G$ such that $I(\chi, \chi^g) \neq 0$ then $g \in N_G(\sqrt[F]{H \cap F})$.

	Suppose $I(\chi, \chi^g) \neq 0$. By Theorem~\ref{thm:kutzko}, we have $I(\chi, \chi^g) \cong \Delta_g \subseteq \Delta$. Therefore by proof of Theorem~\ref{thm:kutzko}, there exists $s \in \Delta$ such that $s(g) \neq 0$. This implies $s(gh^i) = s(g) \chi(h^i)$ is non-zero for all $i$. By the given isomorphism between $\Delta$ and $\mathrm{End}_G(\ind_H^G(\chi))$, as given in proof of Theorem~\ref{thm:kutzko}, we have $s = s_\psi$ for some $\psi \in  \mathrm{End}_G(\ind_H^G(\chi))$. This infers, $\psi f^w(g h^i) \neq 0$ for all $i$. But $\psi f^w \in \ind_H^G(\chi)$ and therefore, by definition of induced representation, is non-zero only for finitely many right cosets of $H$ in $G$. From this it follows that the sets $Hg h^i$ can not be mutually distinct right cosets of $H$ in $G$ for all $i$. Hence, for $h \in H$ and $g \in G$, there exists $\kappa(h,g)$, depending on $h$ and $g$, such that
	$H gh^{\kappa(h,g)} = Hg$. Therefore $g h^{\kappa(h, g)}g^{-1} \in H$.
	
	Let $F$ be the Fitting subgroup of $G$, by Lemma~\ref{fitting-subgroup}, there exists $k$ such that $|H/(H \cap F)| = k$. Consider $ u \in \sqrt[F]{H \cap F}$. By Lemma~\ref{fitting-subgroup} and Lemma~\ref{lem: baumslag}, there exists $t$ such that $u^t \in H \cap F \subseteq H$.  Combining this all, we obtain $(g(u^{tk})g^{-1})^{\kappa(u^t, g)} \in H \cap F$ and therefore $g \in N_G(\sqrt[F]{H \cap F})$.
	
\end{proof}
%
%


%
%


\section{Existence of Schur irreducible induced representation}
\label{sec: existence}
Let $\rho$ be an irreducible
representation of a discrete supersolvable group $G$ of finite weight. In this section, we prove existence of
a subgroup $H'$ and $\chi': H' \rightarrow \mathbb C^{\times}$
such that the following hold
\begin{enumerate}
	\item $V_{H'}(\chi') \neq 0$.
	\item $\ind_{H'}^{G'}(\chi') $ is Schur irreducbile.
\end{enumerate}

Firstly, we dispose of the case of finite dimensional irreducible representations of supersolvable groups.
\begin{proposition}
\label{prop: finite-dimensional}
Every finite dimensional irreducible representation of a discrete supersolvable group is monomial.
\end{proposition}
\begin{proof} This proof is on the lines of an analogous result of finite supersolvable groups (Serre~\cite[Ch.8]{MR0450380}) and is motivated by the analogous result for unitray representations of finitely generated discrete nilpotent groups as given by Brown~\cite[Lemma~1]{MR0352324}. Let $G$ be a discrete supersolvable group and let $H$ be an abelian normal non-central subgroup of $G$. The existence of such group is justified by taking largest $i$ in the invariant series ~(\ref{eq:invarint-series}) of $G$ such that $N_i$ is not contained in the centre of $G$. Then $N_i$ is a normal abelian non-central subgroup of $G$.

Let $(\pi, V)$ be a finite dimensional complex irreducible representation of $G$. Without loss of generality, we can further assume that $\pi$ is faithful. Representation $\pi$ is finite dimensional, therefore there exists a character $\chi$ of $H$ such that $\mathrm{Hom}_H(\pi, \chi) \neq 0$. Consider the subspace $V_H(\chi)$ of $V$. Note that $V_H(\chi)$ is a proper non-trivial subspace of $V$. Let $S = \{g \in G \mid \chi^g = \chi \}$ be the stabilizer subgroup of $\chi$. We note that $S$ is a proper subgroup of $G$ and the group $S$ acts on $V_H(\chi)$. Further any $g \in G \setminus S$ maps $V_H(\chi)$ to $V_H(\chi^g)$ and $V_H(\chi^g) \neq V_H(\chi)$.

We claim that $ V \cong \ind_S^G(V_H(\chi))$ and $V_H(\chi)$ is irreducible as $S-$representations. Then induction hypothesis will imply the result because dimension of $V_H(\chi)$ is strictly less than that of $V$. For this, we note that because $V$ is finite dimensional index of $S$ in $G$ is finite. Therefore $\ind_S^G(V_H(\chi)) \cong \oplus_{g \in G \setminus S} V_H(\chi^g)$ and this is clearly a sub-representation of $V$. The representation $V$ is irreducible, therefore $ V \cong \ind_S^G(V_H(\chi))$. This in turn also implies $V_H(\chi)$ is an irreducible $S$-representation because $V$ is an irreducible $G$-representation.

\end{proof}

Now onwards we concentrate on the case when $\pi$ is an infinite dimensional irreducible representation having finite weight with respect to $(H, \chi)$. Recall that the $G$-space $\ind_{H}^{G}(\chi) $ is called Schur irreducible if $\mathrm{End}_G(\ind_{H}^{G}(\chi)) \cong \mathbb C$. In view of Theorem~\ref{thm: reduction-to-normalizer}, our aim is to find a subgroup $H'$ and its character $\chi'$ such that for all $g \in N_G(\sqrt[F]{H'\cap F}) \setminus H'$ we have $\chi'^g \neq \chi'$ on $H'^g \cap H'$. If pair $(H, \chi)$ itself satisfies this condition then we are done. Otherwise we will modify the pair $(H, \chi)$ in the following steps.
	\subsection{STEP-I} Suppose that there exists $g \in N_G(\sqrt[F]{H \cap F}) $ such that $\chi^g = \chi$ on $H \cap H^g$ and $g^i \notin H$ for any $i$.
	
	By Lemma~\ref{lem: baumslag}, the group $H \cap F$ is a finite index subgroup of $\sqrt[F]{H \cap F}$. Therefore there exists high enough of power of $g$, say $g^k$, such that $(H \cap F)^{g^k} = H \cap F$. We replace $g^k$ with $g$ to obtain $(H \cap F)^g = H \cap F$ and $\chi^g|_{H \cap F} = \chi|_{H \cap F}$.

	Now consider the group $H_1 = \langle H,g \rangle $, the group generated by $H$ and $g$. By the choice of $g$, we have $H_1/(H \cap F)$ is an infinite solvable subgroup. Therefore by Lemma~\ref{lem: infinite-cyclic} there exists a normal subgroup $H_2$ of $H_1$ such that $H_2/(H \cap F)$ is an infinite cyclic normal subgroup. Let $x(H\cap F)$ be a generator of $H_2/(H \cap F)$. Then for all $h \in H$, we have $(hxh^{-1} )(H \cap F)$ equals either $x (H \cap F)$  or $x^{-1} (H \cap F)$. This in turn implies that $xhx^{-1}(\hf)$ equals $h \hf$ or $(x^2h)H \cap F$.
	Therefore every element of the group $\hirschhigh{H}{x}$, group generated by $H$ and $x$, can be written as $x^ih$ for some $i \in \mathbb Z$ and $h \in H$. Collecting all these, we obtain that $\hirschhigh{H}{x}$ is an infinite supersolvable group with Hirsch length exactly one more than that of $H \cap F$. By Lemma~\ref{lem:index dihedral}, there exists a normal subgroup $N$ of $\hirschhigh{H}{x}$ such that $\hirschhigh{H}{x}/N$ is either an infinite cyclic or an infinite Dihedral group.

\begin{lemma}
		The group $HN/N$ is a finite abelian subgroup of $\hirschhigh{H}{x}/N$ such that either $|HN/N| = 1$ or $|HN/N| = 2$.
\end{lemma}
	\begin{proof} We have $(H \cap F) \subseteq (H \cap N)$ and $(H \cap F)$ is a finite index normal subgroup of $H$. Now the finiteness of $HN/N$ follows because $HN/N \cong H/(H \cap N)$. If $\hirschhigh{H}{x}/N$ is infinite cyclic then $HN/N$ is trivial and when $\hirschhigh{H}{x}/N$ is infinite dihedral then either $HN/N$ is trivial or $HN/N$ is a group of order two.
	\end{proof}	
We continue with the construction of $(H', \chi').$  We also note that by definition of $N$ and $\hirschhigh{H}{x}$, the group $H\cap F$ is  normal in $N $ such
that $\h(N) = \h(H \cap F)$. Therefore index of $H \cap F$ in $N$ is finite. \\

{{\bf Case 1: $ \bf |HN / N| = 1$}}

	\label{sec:hn=n}  In this case, we have $H \cap N = H$ and therefore $H \subseteq N$. Further because $H \cap F \subseteq H \subseteq N$, index of $H$ in $N$ is finite. By Lemma~\ref{finite-finite.index}, there exists $r$ such that $H = H^{x^r}$. Therefore the infinite group generated by $x^r$ acts on irreducible representations of $H$ and $\chi = \chi^{x^{ri}}$ for all $i$ on $H \cap F$. By Proposition~\ref{prop: reduction}, we have 
	\[
	\mathrm{End}_H( \ind_{H\cap F}^H(\chi|_{H \cap F}) , \chi^{x^{ri}}) \neq 0 \,\,\forall \,\, i \geq 0.
	\] But $\ind_{H \cap F}^H(\chi|_{H \cap F})$ is finite dimensional and therefore $\chi = \chi^{x^{rj}}$ for some $j$.
	Let $K$ be the group generated by $x^{rj}$. The space $V_H(\chi)$ is finite dimensional and invariant under the action of $K$. Therefore there exists a character $\delta$, such that
	\[
	V_{H}(\chi) \cap V_{K}(\delta) \neq 0.
	\]
	We apply Lemma~\ref{diamond-lemma} for this $H$, $\chi$, $K$, $\delta$ and obtain a character
	$\chi'= \chi \delta $ of group $H' = HK$. we note that because $\chi'|_H = \chi$, $V_{H'}(\chi')$ is finite dimensional. Further $\h(H') > \h(H)$.\\
	
{\bf Case 2: $\bf |HN / N| = 2$}

As $HN/N \cong H/(H \cap N)$, we have $ H \cap N$ is an index two subgroup of $H$. Since $H \cap F \subseteq H \cap N$, so $H \cap N$ is also a finite index subgroup of $N$. Again by the same arguments as in Case-1, we can replace $x$ by some high power so that $H \cap N$ is normal in $\hirschhigh{H}{x}$ and $\chi^x|_{H \cap N} = \chi|_{H \cap N}$.
	
	Let $h \in H \setminus H \cap N$ be such that $h^2 = h'\in H \cap N$. If, $hxh^{-1} (H \cap F) = x(H \cap F)$ then it is easy to see that we can assume that $H$ is a normal subgroup of $\langle H, x \rangle$. Then as in Case 1, we can extend $\chi$ to $H' = \langle H, x \rangle$. We denote this extension by $\chi'$. We have $\chi'|_H = \chi$, therefore $V_{H'}(\chi')$ is finite dimensional. Further $\h(H') > \h(H)$.

	To deal with the remaining case we will assume that $hxh^{-1} (H \cap F) = x^{-1}(H \cap F)$. Consider the space $V_{H \cap N}(\chi|_{H \cap N})$ and define the following subspaces.
	\[
	W_\pm = \{v \in V_{H \cap N}(\chi|_{H \cap N}) \mid \pi(h)v = \pm \sqrt{\chi(h')}v\}.
	\]
	Then, we have the following.
	\[
	V_{H \cap N} (\chi|_{H \cap N}) = W_+ \oplus W_-.
	\]
	Therefore any vector $v$ of $V_{H \cap N} (\chi|_{H \cap N})$ is uniquely written as $v = w_+ + w_-$.
	Further by the given hypothesis on $V_H(\chi)$, either $W_+$ or $W_-$ is finite dimensional. Without loss of generality assume $W_-$ is finite dimensional. Argument for $W_+$ is similar. Let $\{ w_1, \ldots, w_r \}$ be a basis of $W_-$. The group $H \cap N$ is normal in $\hirschhigh{H}{x}$ and $\chi|_{H \cap N}$ is invariant under the action of $x$. Therefore $V_{H \cap N} (\chi|_{H \cap N})$ is invariant under  the action of $x$.
	\begin{lemma}
		\label{lem:finite-hg}
		There exists a subspace $Y$ of $V_{H \cap N} (\chi|_{H \cap N})$ such that $Y$ is finite dimensional and $\hirschhigh{H}{x}$-invariant. 
	
	\end{lemma}
  \begin{proof} Let $\{n_1=1, n_2, \cdots, n_t \}$ be the set of coset representatives of $H \cap N$ in $N$.
		Consider finite dimensional subspace $Y$ of $V_{H \cap N} (\chi|_{H \cap N})$ generated by $$\{w_j, (\rho(n_i)w_j)_+, (\rho(xn_i)w_j)_+ \}$$ 
for $1 \leq i \leq r$ and $1 \leq j \leq t$. The finite dimensional space $Y$ is clearly $H$-invariant. For action under $x$, write
		\begin{equation}
	\label{eq:interplay-x-h}
		\rho(x) \rho(n_j) w_i = -\frac{1}{\sqrt{\chi(h')}} (\rho(x) \rho(n_j)\rho(h)) w_i,
		\end{equation}
		Let $n_k$ be a coset representative of $H\cap N$ in $N$ and $h_k \in H \cap N$ such that  $xh(h^{-1}n_j h) = hx^{-1}n_k h_k$. Then we obtain the following.
		\begin{equation}
		\label{eq:interplay-x-h2}
			(\rho(x) \rho(n_j)\rho(h)) w_i = \rho(h) \rho(x^{-1})\rho(n_k) \rho(h_k) w_i
		\end{equation}
Now using,
		\begin{equation}
		\rho(x) \rho(n_j )w_i=  (\rho(x) \rho(n_j) w_i)_+ +  (\rho(x) \rho(n_j) w_i)_-,
		\end{equation}
		and (\ref{eq:interplay-x-h}), (\ref{eq:interplay-x-h2}), we obtain the following.
		\begin{equation}
\rho(n_k) \rho(h_k) w_i  = \sqrt{\chi(h')}\left [ \rho(x)\left ( (\rho(x) \rho(n_j) w_i)_+ \right ) - \rho(x) \left ((\rho(x) \rho(n_j) w_i)_- \right )\right ].
		\end{equation}
		We have $(\rho(x) \rho(n_j) w_i)_- \in W_-$ so can be written as linear combination of $w_i$'s and further $\rho(h_k) w_i$ for $h_k \in H \cap N$ is a scalar multiple of $w_i$. Therefore we have, $\rho(x)( (\rho(x) \rho(n_j) w_i)_+ ) \in Y$.
	\end{proof}
	Since $Y$ is a finite dimensional $\hirschhigh{H}{x}$-invariant subspace of $V_{H \cap N} (\chi|_{H \cap N})$. Therefore there exists a vector $v \in Y$ and character $\tilde{\chi}: \hirschhigh{H \cap N}{x} \rightarrow \mathbb C^\times$ such that,
	\begin{equation}
	\label{eq:hng-invariant}
	\rho(z) v = \tilde{\chi}(z) v \,\, \forall \,\, z \in \hirschhigh{H \cap N}{x}.
	\end{equation}
	Consider the subspace $Y_0$ of $Y$ generated by $\{v, \rho(h)v \}$, where $v$ is as defined in (\ref{eq:hng-invariant}). Then $Y_0$ is $\hirschhigh{H}{x}$ invariant subspace with action of $h$ given by the following.
\[
\rho(h)(\rho(h)v) = \chi(h')v.
\]
Therefore it is easy to see that both $(Y_0)_+$ and $(Y_0)_-$ are non-trivial. We denote the above action of $\hirschhigh{H}{x}$ on $Y_0$ by $\rho_0$. Define the space $V_{\hirschhigh{H}{x}}( \rho_0)$ as follows.
	\[
	V_{\hirschhigh{H}{x}}(\rho_0) = \{ v \in V \mid \rho(z) v = \rho_0(z)v\,\, \mathrm{for\,\,all}\,\, z \in \hirschhigh{H}{x} \}
	\]
	\begin{lemma}
		\label{lem:W0-finite-multiplicity}
		The space $V_{\hirschhigh{H}{x}}( \rho_0)$ is finite dimensional.
	\end{lemma}
	\begin{proof} By definition of $\rho_0$, we have $V_{\hirschhigh{H}{x}}( \rho_0) \subseteq V_{H \cap N} (\chi|_{H \cap N})$. Therefore, we can decompose $V_{\hirschhigh{H}{x}}( \rho_0)$ as follows.
		\[
		V_{\hirschhigh{H}{x}}( \rho_0) =  V_{\hirschhigh{H}{x}}( \rho_0)_+ \oplus V_{\hirschhigh{H}{x}}( \rho_0)_-.
		\]
		We have $V_{\hirschhigh{H}{x}}( \rho_0)_- \subseteq W_-$ and therefore is finite dimensional. To prove the claim, we need to prove that the space $V_{\hirschhigh{H}{x}}( \rho_0)_+$ is finite dimensional. We prove this by contradiction. In particular we prove that if $V_{\hirschhigh{H}{x}}( \rho_0)_+$ is infinite dimensional then there exists vector spaces $U_i$ satisfying the following.
		\begin{enumerate}
			\item Space $U_i$ are finite dimensional $\hirschhigh{H}{x}$-invariant.
			\item $U_i \subsetneq U_{i+1} \subsetneq V_{\hirschhigh{H}{x}}( \rho_0)$ for all $i$.
			\item $\emptyset \neq U_i \cap W_- \subsetneq U_{i+1} \cap W_-$.
		\end{enumerate}
		But then this contradicts the finite dimensionality of $W_-$.
		The spaces $U_i$ are constructed inductively. Suppose $U_i$ is already constructed then we construct $U_{i+1}$ as follows. Choose vector $u_{i+1}$ such that $u_{i+1}  \in   V_{\hirschhigh{H}{x}}( \rho_0)_+ \setminus( U_{i})_+$ and take $U_{i + 1}$ to be space generated by $U_i$ and set $S = \{u_{i+1}, (\rho(h)u_{i+1}) \}$. As mentioned earlier, the space generated by $S$ is $\hirschhigh{H}{x}$ invariant and therefore (1) and (2) follows. For (3) we see that the space generated by $S$ is $\hirschhigh{H}{x}$-invariant and as above intersects $W_-$ non-trivially. This completes the proof.
	\end{proof}
	In case $\rho_0$ is one dimensional, then we take $H' = \hirschhigh{H}{x}$ and $\chi' = \rho_0$. In case $\rho_0$ is two dimensional, we have that $\rho_0 \cong \ind_{\hirschhigh{H \cap N}{x}}^{\hirschhigh{H}{x}}(\tilde{\chi})$. Now since   $V_{\hirschhigh{H \cap N}{x}}(\tilde{\chi}) \subseteq V_{\hirschhigh{H}{x}(\rho_0)}$. By Lemma~\ref{lem:W0-finite-multiplicity}, latter is finite dimensional and therefore $V_{\hirschhigh{H \cap N}{x}}(\tilde{\chi})$ is finite dimensional.
%
%
%
	In any case we see that for obtained tuple $(H', \chi'),$ we have $V_{H'}(\chi') \neq 0$ is finite dimensional and Hirsch length of $H'$ is strictly greater than that of $H$.
	From the proof actually a stronger result follows.
	\begin{thm}
		\label{thm: infinite-g} Suppose $V$ is an irreducible representation of a discrete supersolvable group $G$ having finite weight.
		Let $(H, \chi)$ be such that $H$ has maximum
		Hirsch length possible among the pairs $(H', \chi')$ with $V_{H'}(\chi')$ is non-trivial and finite dimensional. Then for any $g \in G $ such that $\chi^g = \chi$ on $H \cap F$, we have $g^i \in H$ for some $i$.
	\end{thm}

	Thus now onwards we assume that $(H, \chi)$ are chosen so that pair $(H, \chi)$ is
	maximal with the property that Hirsch length of $H$ is maximum and $V_H(\chi) \neq 0$ is finite dimensional. If for this maximal $(H, \chi)$, there exists $g \in N_G(\sqrt[F]{H \cap F}) \setminus H$ such that $\chi^g = \chi$ on $H \cap H^g$, we will modify our pair $(H, \chi)$ further as given in next step.

	\subsection{STEP-II}
	For the next step, let $H_0 = \cap_{N_G (\sqrt[F]{H \cap F})} gHg^{-1}$ and $\chi_0 = \chi|_{H_0} $. By Lemma~\ref{lem: baumslag} and Lemma~\ref{finite-finite.index}, we have $H_0 \cap F$ is a finite index subgroup $H \cap F$ and therefore also has finite index in $H$. Let,
	
	\[ L = \{ g \in G \mid \chi^g = \chi \,\, \mathrm{on} \,\, H^g \cap H \}.
	\]
	By Theorem~\ref{thm: reduction-to-normalizer} and Lemma~\ref{lem: baumslag}, we have $L \subseteq N_G(\sqrt[F]{H \cap F})$. By definition, the group $H_0$ is normal in $N_G(\sqrt[F]{H \cap F})$. Define,
	\[
	L_0 = \{ g \in N_G(\sqrt[F]{H \cap F}) \mid \chi_0^g= \chi_0 \}.
	\]
	
	Then  $ H_0 \subseteq H \subseteq L \subseteq L_0$ and $H_0$ is a normal subgroup of $L_0$. If $H_0 = L_0$, then we are done.
	We claim that $L_0/H_0$ is a finite group. For if there exists $g \in L_0$ such that $\chi_0^g = \chi_0$ and $g^i \notin H_0$ for any $i$. Then again as in the Case 1, we can replace $g$ with some high power to obtain that $\chi^g = \chi$ on $H \cap F$. But then $g^i \notin H$ for any $i$ contradicts maximality of $(H, \chi)$ by Theorem~\ref{thm: infinite-g}. This implies that $L_0/H_0$ is super-solvable  group finitely generated by finite order elements which does not contain an element of infinite order. Therefore $L_0/H_0$ is finite by Lemma~\ref{lem: infinite-cyclic}. Let $H_1/H_0$ be a cyclic normal subgroup of $L_0/H_0$ and let $xH_0$ be a generator of $H_1/H_0$. Let $S
	= (x)$ be the cyclic subgroup generated by $x$. Then $x$ normalizes $H$, therefore there exists a character $\delta: S
	\rightarrow \mathbb C^{\times}$ such that
	\[
	V_{H_0}(\chi_0) \cap V_S(\delta) \neq 0.
	\] Then by
	Lemma~\ref{diamond-lemma} for these $H_0$, $\chi_0$, $S$, $\delta$, we get that there exists a character
	$\chi_1$ of $H_1 = H_0S$ that extends $\chi$ and $V_{H_1}(\chi_1) \neq 0$.
	
	Let $L_1 = \{ g \in L_0 \mid (\chi_1)^g = \chi_1\}$. Now $L_1/H_1$ is also a finite subgroup of order strictly smaller than that of $L_0/H_0$ and therefore
	continuing this way, there exists a subgroup $H_k$
	and a character $\chi_k$ of $H_k$ such that $H_k = L_k$, $(\chi_k)|_{H_0} = \chi_0$, $V_{H_k}(\chi_k) \neq 0$,
	and $(\chi_k)^g \neq \chi_k$ on $H_k^g \cap H_k$ for any $g \in N_G(\sqrt[F]{H_k \cap F}) \setminus H_k$.
	The last assertion follows by the fact that $\sqrt[F]{H \cap F} = \sqrt[F]{H_k \cap F}$.
	We denote this obtained pair $(H_k, \chi_k)$ by $(H', \chi')$ and this pair is such that $\ind_{H'}^G(\chi')$ is Schur irreducible.

\section{Irreducibility of induced representation}
\label{sec: irreducible}
In this section we prove that for a supersolvable group $G$, if induced representation $\ind_H^G(\chi)$ is Schur irreducible then it is infact an irreducible representation.
The proofs of this section are modelled on ~\cite[Section~3]{Beloshapka-Gorchinskiy}.


First of all we prove the following general result which is true for all discrete groups.
\begin{thm}
	\label{thm:proximate-schur}
Let $H$ be a proximate normal subgroup of a discrete group $G$ and let $(\rho, W)$ be an irreducible representation of $H$ such that $\tilde{\rho} = \ind_H^G(\rho)$ is Schur irreducible. Then the induced representation $\ind_H^G(\rho)$ is irreducible.
\end{thm}
\begin{proof}
To prove that $\ind_H^G(\rho)$ is irreducible, we show that it is generated by any non-zero $v \in \ind_H^G(\rho)$ as $G$-space.

Let $H$ be a subgroup of $G$. Then by Mackey's formula for compact induction (see~(\ref{eq:mackey-formula})), we have the following.
\begin{equation}
\label{eq:induction-reduction}
\ind_H^G(\rho)|_H \cong \oplus_{g \in H \backslash  G /H  } \ind_{H \cap H^g}^H(\rho^g|_{H \cap H^g}).
\end{equation}
By the given hypothesis, $H \cap H^g$ is normal subgroup of $H$ and therefore we have the following.
\begin{equation}
\label{eq:reduction-to-H-subgroup}
\ind_{H \cap H^g}^H(\rho^g)|_{H \cap H^g} \cong \oplus_{h \in H /(H \cap H^g)} \rho^{hg}.
\end{equation}
Thus by (\ref{eq:induction-reduction}) and (\ref{eq:reduction-to-H-subgroup}), every vector $v \in \ind_H^G(\rho)$ can be written as the following.
\begin{equation}
\label{eq:v-form}
v  =   v_1 + v_2 + \ldots + v_{k(v)},
\end{equation}
 such that $v_i \in  \ind_{H \cap H^{g_i}}^H(\rho^{g_i})$ are non-zero. Therefore
 \begin{equation}
 \label{eq:vi-form}
 v_i = \sum_{j = 1}^{l_i(v)} \tilde{\rho}(h_j g_i) w_j,
 \end{equation}
 for $h_j \in H/(H \cap H^{g_i})$ and non-zero $w_j \in W$. We note that if $v = w$ for some $w \in W$, then $\tilde{\rho}(g)v$ for all $g \in G$ is contained in the $G$-space generated by $v$. Therefore from (\ref{eq:v-form}) and (\ref{eq:vi-form}),  it is clear that any $v$ such that $v = w$ generates $\ind_H^G(\rho)$.

 We use induction on $\sum_{i =1}^{k(v)} \ell_i(v)$ to prove the result. For the case when $\sum_{i =1}^{k(v)} \ell_i(v)  = 1$, then $k(v) = \ell_i(v) = 1$, therefore $v = \tilde{\rho}(hg)w$ for some $g \in H \backslash G /H$, $h \in H/(H \cap H^g)$ and non-zero $w \in W$. By multiplying $v$ with $\tilde{\rho}(g^{-1}h^{-1})$ we obtain that $w \in W$ is in the $G$-space generated by $v$. Then as mentioned earlier $v$ generates the whole space $\ind_H^G(\rho)$.

 For the case when $\sum_{i =1}^{k(v)} \ell_i(v)  > 1$, by multiplying with suitable $\tilde{\rho}(g)$ for $g \in G$, we can assume that $v_1$ is such that $\tilde{\rho}(g_1)$ is identity and therefore $v_1 \in W$. Thus we have the following.
 \[
 v = \sum_{i=1}^{k(v)} \sum_{j = 1}^{\ell_i(v)} \tilde{\rho}(h_{ij}g_i)w_{ij},
 \]
 with $g_1$ and $h_{11}$ equal to the identity and $\ell_1(v) = 1$. By multiplying with suitable element of $h$, we can further assume that $h_{21}$ is identity. Note that this preserves the sum $\sum_{i =1}^{k(v)} \ell_i(v) $. Thus $v$ is of the following form form.
 \[
 v = w_{11} + \tilde{\rho}(g_2) w_{21} + \tilde{\rho}(h_{22} g_2) w_{22} + \cdots + \tilde{\rho}(h_{k(v)\ell_{k(v)}(v)} g_{k(v)})w_{k(v)\ell_{k(v)}(v)}.
 \]

 Let $I$ be the kernel of $w_{11}$ under the $\mathbb C[H]$ module action on $W$. The action of $H$ on $W$ is irreducible. Therefore,
  \[\mathbb C[H]/I \cong \mathbb C[H]w_{11} = W.
  \]
  Let $ I_0 = I \cap \mathbb C[H \cap H^{g_2}]$. Recall, the algebra $\mathbb C[H]$ consists of linear space generated by finite support functionals on $H$. We identify $\mathbb C[H \cap H^{g_2} ]$ with the subspace of $\mathbb C[H]$ consisting of linear space generated by functionals with their support zero outside $H \cap H^{g_2}$. Then, it is easy to see that there exists a well defined non-trivial $\mathbb C[H \cap H^{g_2}]$-module map from $\mathbb C[H]/I$ to $\mathbb C[H \cap H^{g_2}]/I_0$ by projecting functions with support outside $H \cap H^{g_2}$ to zero. This is depicted as below.
 \[
 \xymatrix{
 \mathbb C[H] \ar@{->>}[r] \ar@{->>}[dr] & \mathbb C[H \cap H^{g_2}] \ar@{->>}[r] & \frac{\mathbb C[H \cap H^{g_2}]}{I_0} \\
 & \frac{\mathbb C[H]}{I} \ar@{->>}[ur]^f &
  	}
 \]

  Let $J$ be the kernel of $w'= \tilde{\rho}(g_2) w_{21} \in \ind_{H \cap H^{g_2}}^H(\rho^{g_2}/(H \cap H^{g_2}))$ under the $\mathbb C[H \cap H^{g_2}]$-module action. Although $\ind_{H \cap H^{g_2}}^H(\rho^{g_2}/(H \cap H^{g_2}))$ is not necessarily an irreducible $\mathbb C[H \cap H^{g_2}]$-module, still the following holds.
  Thus
  \begin{equation}
  \label{eq: wg2-module}
  \mathbb C[H \cap H^{g_2}] w' \cong \mathbb C[H \cap H^{g_2}]/J.
  \end{equation}
    We claim that $I_0 \neq J$. As otherwise there exists a non-trivial $H \cap H^{g_2}$-linear map between $W$ and $W^{g_2}$ given by the following.
\[
\xymatrix{
W \ar@{->>}[r] & \mathbb C[H \cap H^{g_2}]/I_0 \ar[r]^\cong  &  C[H \cap H^{g_2}]/J  \ar@{^(->}[r] & W^{g_2} }.
\]
In the above diagram, last inclusion is true because $\mathbb C[H \cap H^{g_2}] w' \subseteq W^{g_2}$ and (\ref{eq: wg2-module}). This implies $\mathrm{Hom}_{H \cap H^{g_2}}(\rho|_{H \cap H^{g_2}}, \rho^{g_2}|_{H \cap H^{g_2}}) \neq 0$. Then by the fact that $H \cap H^{g_2}$  has finite index in $H$ and (\ref{frobenius-reciprocity}), we have that the space
$\mathrm{Hom}_H(\rho, \ind_{H \cap H^{g_2}}^H(\rho^{g_2}))$ is non-zero. This is a contradiction to the fact that $\ind_H^G(\rho)$ is Schur irreducible (see Remark~\ref{rk:general-intertwiner}).

Therefore there exists $x \in \mathbb C[H]$ such that $x \in (I_0 \setminus J) \cup (J \setminus I_0)$. We get for any $h \in H$ and $g \in G$, $x h g = h_1 \tilde{g} h_2$ for some $h_2 \in H$, $\tilde{g} \in H \backslash G / H$ and $h_1 \in H/(H \cap H^{\tilde{g}})$. At this moment induction will apply because $\sum_{i =1}^{k(xv)} \ell_i(xv) < \sum_{i =1}^{k(v)} \ell_i(v) $.
 \end{proof}

The following generalizes the notion of a subnormal group.
\begin{definition} (Subproximate normal series) A chain of subgroups of $H$ given by the following,
	\begin{equation}
	\label{eq:subproximate-series}
	G = H_k \supsetneq H_{k-1} \supsetneq \cdots \supsetneq H_0 = H,
	\end{equation}
	is called Subproximate normal series of length $k$, if $H_i$ is proximate normal in $H_{i+1}$ for all $0 \leq i \leq k-1$.
\end{definition}
\begin{definition} (Subproximate normal subgroup)
\label{defn:subproximate-normal} A subgroup $H$ of a group $G$ is called subproximate normal subgroup if it has a finite subproximate normal series.
\end{definition} 
It is well known that every subgroup of a finitely generated nilpotent group is subnormal. The following lemma shows that the notion of subproximate normal is an analogous notion for supersolvable groups.
\begin{lemma}
\label{proximate-series}	
Every subgroup of a supersolvable group is subproximate normal subgroup.
\end{lemma}
\begin{proof}
Let following be the lower central series of Fitting subgroup $F$ of supersolvable group $G$.
\[
F = \gamma_0(F) \supset \gamma_1(F) \supset \cdots \supset \gamma_k(F) = [F, \gamma_{k-1}(F)] \supset \gamma_{k+1}(F) = (1)
\] 	
For a given subgroup $H$ of $G$, we consider the following chain of groups.
\[
G \supset HF \supset H\gamma_1(F) \supset \cdots \supset H\gamma_k(F) \supset H\gamma_{k+1}(F) = H \supset (1).
\]
To show that $H\gamma_i(F)$ is proximate normal in $H\gamma_{i-1}(F)$ for all $ 1 \leq i \leq k+1$, consider finite index normal subgroup $(H \cap F) \gamma_i(F)$ of $H \gamma_i(F)$. The groups $\gamma_i(F)$ are characteristic subgroup of normal subgroup $F$ of $G$, therefore are normal in $G$. It follows that $(H \cap F) \gamma_i(F)$ is normal in $H\gamma_{i+1}(F)$.

 By definition of $F$, we have $H \gamma_i(F)/((H \cap F) \gamma_i(F))$ is a finite abelian group. For all $g \in H\gamma_{i+1}(F)$, we have  $((H \cap F) \gamma_i(F))$ is a subgroup of $(H \gamma_i(F)) \cap (H \gamma_i(F))^g$. Therefore $(H \gamma_i(F)) \cap (H \gamma_i(F))^g$ is a finite index normal subgroup of $H$ for all $g \in H\gamma_{i+1}(F)$. For the case of subgroup $HF$ of $G$, result follows because $F \subseteq HF \subseteq G$ and $G/F$ is a finite abelian group.
\end{proof}	
\begin{proof}[Proof of Theorem~\ref{thm: supersolvable-irreducible}] By Lemma~\ref{proximate-series}, every subgroup of a supersolvable group has a finite subproximate normal series. We use induction on minimum length of a subproximate normal series of $H$. In case $H$ is proximate normal subgroup of $G$, then result follows by Theorem~\ref{thm:proximate-schur}. Let minimum length of a subproximate normal series of $H$ is $k$. By induction, we assume that result is true for all subgroups with a subproximate normal series of minimum length strictly less than $k$. Let subproximate normal series of $H$ of length $k$ be given as below.
\[
G = H_k\supsetneq H_{k-1} \supsetneq \cdots \supsetneq H_0 = H. 	
\]
Then $\ind_H^{H_{1}}(\chi)$ is irreducible by Theorem~\ref{thm:proximate-schur}. Further $\ind_{H_{1}}^G(\ind_H^{H_{1}}(\chi))$ is isomorphic to $\ind_H^{G}(\chi)$; therefore is Schur irreducible. The group $H_1$ has a subproximate normal series of length $k-1$ and therefore result follows by induction.

\end{proof}

\section{Proof of Theorem~\ref{main-theorem}}
\label{sec: main-theorem}
In this section we prove the main theorem. \\
\begin{proof}[Proof of Theorem~\ref{main-theorem}]
Firstly we prove the necessary part. For given irreducible representations $(\pi, V)$ with finite weight, let $(H', \chi')$ be obtained by section~\ref{sec: existence}. Then $V_{H'}(\chi') \neq 0$ and $\ind_{H'}^G(\chi')$ is Schur irreducible, which infact is irreducible by Section~\ref{sec: irreducible}. Then $I(\mathrm{Ind}_{H'}^G(\chi'), \pi) \neq 0$. By Proposition~\ref{prop: reduction}, there exists a non-trivial $G$-linear map from $\mathrm{Ind}_{H'}^G(\chi')$ to $\pi$. But then both of these representations are irreducible and therefore $\mathrm{Ind}_{H'}^G(\chi') \cong \pi$. This prove that $\pi$ is monomial.

Conversely, Let $\pi \cong \mathrm{Ind}_{H'}^G(\chi')$ be a monomial irreducible representation
of $G$ acting on representation space $V$. We prove that
$V_{H'}(\chi')$ is in fact one
dimensional subspace of $V$ and thus proving that $(\pi, V)$ has finite weight. Let $\{g_i \mid i \in \mathbb I \}$ be a set of double coset
representatives of $H'$ in $G$. For each $i \in \mathbb I $, define
functions $f_i$ such that $f_i(g_j) = \delta_{i,j}$ and then
extended to the whole group $G$ so that $f_i(hg_j) =
\chi'(h)f_i(g_j)$ for all $j$. From the definition of
$\mathrm{Ind}_{H'}^G(\chi')$, it is clear that every $f \in V$ can be
written as linear combination of $f_i$'s. The representation $V$
is countable dimensional and irreducible; therefore, also Schur irreducible. By Schur's lemma and Theorem~\ref{thm: reduction-to-normalizer} for any $g \notin H'$, there exists $h \in H' \cap {H'}^g$ such that $\chi(h) \neq
\chi^g(h)$.  For any $f \in V$, we have the following. 
\[
\pi(h)f(g) = f(ghg^{-1}g) = \chi^g(h) f(g) \neq \chi(h)f(g),
\]
therefore $f \in V_{H'}(\chi')$ if and only if $f$ is non-zero only on
trivial coset representative of $H'$ in $G$ and here it is
determined by its value on identity element of $H'$. Therefore
$V_{H'}(\chi')$ is one dimensional.

\end{proof}

\section{Infinite Dihedral Group}
\label{sec: infinite-dihedral}
In this section we show that every irreducible representation of infinite Dihedral group $D_\infty$ is monomial. Recall $D_\infty$ is semidirect product of $\mathbb Z_2 = (s)$ with infinite cyclic group $(g)$, with $\mathbb Z_2$ action given by $sgs^{-1} = g^{-1}$.

\begin{thm}
Every irreducible representation of $D_\infty$ is monomial
\end{thm}
\begin{proof}
For group $D_\infty$, we prove that every irreducible representation is finite dimensional and therefore by Proposition~\ref{prop: finite-dimensional}  is monomial. We prove this by contradiction.

Let $\pi: D_\infty \rightarrow \mathrm{Aut}(V)$ be an  infinite dimensional irreducible representation of $D_\infty$.
Consider non-zero $v \in V$, and let $W$ be the space generated by $g^i v$ for all $i \in \mathbb Z$. Now $V$ is irreducible, therefore $V = W + s(W)$. So we have either $sW = W$ (and therefore $V =W$) or $W \subsetneq V$. In the latter case, we have $s(g^i v) \notin W$ for any $i \in \mathbb Z$ as otherwise $sv \in W$.

{\bf Case 1: $\bf {sW \cap W = (0)}$} For this case, we first note that the sets $\{ g^iv\}_{i \in \mathbb Z}$ and $\{sg^i v\}_{i \in \mathbb Z}$ form basis of $W$ and $sW$ respectively. Consider the proper subspace $W_0$ of $W$ defined as below.
\begin{eqnarray}
\label{eq:w0}
W_0 = \{ \sum_{i} a_i(g^iv) \in W \mid \sum_i a_i = 0 \}
\end{eqnarray}
Then $W_0 + s(W_0)$ is a proper $D_\infty$-invariant subspace of $V$ and this contradicts irreducibility of $V$.

{\bf Case 2: ${\bf 0 \subsetneq sW \cap W \subsetneq W} $ } In this case, there exists non-zero $w, w' \in W$ such that $w = sw'$. Consider $W_1$, subspace of $W$, generated by $g^i w$ for all $i$. Then $W_1 + sW_1$ is a a $D_\infty$-invariant subspace of $sW \cap W$ and therefore proper subspace of $V$ contracting the irreducibility of $V$.

{\bf Case 3: $\bf{sW = W}$} For this case, we denote $g^i v$ by $v_i$ for all $i \in \mathbb Z$. Then $s v_i = s(g^{i} v_0) = g^{-i} s(v_0)$. Therefore action of $s$ on $W$ is determined by its action on $v_0$. Let $s(v_0) = \sum_{j = r}^t a_j v_j$ such that $ r \leq t$ and $a_{r}$, $a_t$ are non-zero. Then we have,
\[
s(s(v_0)) = s( \sum_{j = r}^t a_j v_j) = \sum_{j = r}^t a_j (\sum_{l = r}^ts a_l v_{l+j})
\]
By concentrating on the smallest and largest index, we obtain
\[
s(s (v_0)) = a_{r}^2 v_{2r}+ \cdots + a_t^2 v_{2t},
\]
with the condition that middle expression does not involve $v_{2r}$ and $v_{2t}$. But we have $s(s(v_0)) = v_0 $, therefore $ r = t = 0$ and $a_r^2 = 1$. So $s(v_0) =  \pm v_0$ and this determines $s$. Again let $W_0$ be a proper subspace of $W$ as defined in (\ref{eq:w0}). Then $W_0$ is a proper $D_\infty$-invariant subspace of $W$, again contradicting irreducibility of $V$. This implies that every irreducible representation of $D_\infty$ is finite dimensional.
\end{proof}

{\bf Acknowledgment.} This work is supported in part by UGC Centre
for Advanced Studies. The second named author thanks Arvind Ayyer, Nutan Limaye, and Amritanshu Prasad for their encouragement.

\bibliography{h}{}
\bibliographystyle{siam}

\end{document}